\documentclass{article}
\usepackage{graphicx} 
\usepackage{amsmath, amssymb, amsthm}
\usepackage{mathtools}
\usepackage{url}

\theoremstyle{definition}
\newtheorem{definition}{Definition}
\newtheorem{theorem}{Theorem}
\newtheorem{lemma}{Lemma}
\newtheorem{proposition}{Proposition}
\theoremstyle{remark}

\numberwithin{equation}{section}

\newcommand{\keyeq}{%
  \refstepcounter{equation}%
  \tag{\theequation$\;\star$}%
}

\title{On the Periodic Orbits of the Dual Logarithmic Derivative Operator}
\author{
  Xiaohang Yu\thanks{Imperial College London}
  \and
  William Knottenbelt\footnotemark[1]
}
\date{November 2025}

\begin{document}

\maketitle

\begin{abstract}
We study the periodic behaviour of the dual logarithmic derivative operator $\mathcal{A}[f]=\mathrm{d}\ln f/\mathrm{d}\ln x$ in a complex analytic setting. We show that $\mathcal{A}$ admits genuinely nondegenerate period-$2$ orbits and identify a canonical explicit example. Motivated by this, we obtain a complete classification of all nondegenerate period-$2$ solutions, which are 
precisely the rational pairs $(c a x^{c}/(1-ax^{c}),\, c/(1-ax^{c}))$ with $ac\neq 0$. We further classify all fixed points of $\mathcal{A}$, showing that every solution of $\mathcal{A}[f]=f$ has the form $f(x)=1/(a-\ln x)$. As an illustration, logistic-type functions become pre-periodic under $\mathcal{A}$ after a logarithmic change of variables, entering the period-$2$ family in one iterate. These results give an explicit description of the low-period structure of $\mathcal{A}$ and provide a tractable example of operator-induced dynamics on function spaces.
\end{abstract}

\section{Introduction}

Expressions of the form  $d\ln y / d\ln x$ arise naturally across mathematics and the applied sciences, most prominently in the concept of \emph{elasticity}, where this log--log rate of change measures the relative responsiveness of $y$ to variations in $x$~\cite{mathcentreElasticityGuide, Tarasova2016ElasticityMemory, Saez2015GovPurchases, Kim2023FirmRevenue}.  
The same dual logarithmic derivative is also widely used to analyze scaling relations in biology and physics \cite{West2005ScalingEcology, Yamasaki2003FluxCreep}.

Yet, despite its broad appearance, the iterative behaviour of this Dual Logarithmic Derivative (DLD) operator has received little, if any, attention.

\paragraph{Dual Logarithmic Derivatives.}
The classical logarithmic derivative $f'/f$ is well established in analysis and differential algebra for describing local growth and zero–pole structure 
\cite{Han2025LogDerivativeLemma, menous2012logarithmicderivativesgeneralizeddynkin, KaraBelaidi2023FastGrowth}.
A related but distinct line of work investigates the operator $\ln(D_x)$, the logarithm of the differentiation operator, whose action on analytic and special functions has been systematically analyzed in operational-calculus settings~\cite{onthelogarithmofaderivativeoperator2025}.

In contrast, this work studies the \emph{dual logarithmic derivative}
\[
A[f](x)\coloneqq\frac{d\ln f}{d\ln x} 
\]
which measures variation under logarithmic scaling of the input.  
While elementary algebraic transformations may relate the dual logarithmic derivative to previously studied forms, the operator itself exhibits distinct periodic dynamical behaviours that appear not to have been examined before.

\paragraph{Function-space Fixed-point Theory.}
In modern analysis, fixed-point theory transcends the classical setting of real-valued functions to address solutions of operator equations in abstract function spaces. A fixed point for an operator \(A\) acting on a function space \(\mathcal{F}\) is a function \(f\in \mathcal{F}\) such that \(A[f]=f\). More generally, the study of periodic points, where \(A^{k}[f]=f\) for some integer \(k\ge 2\), serves as a powerful framework for characterizing the long-term behaviour of nonlinear systems. \cite{Kuczma1990FunctionalEqns,Elaydi2005DifferenceEqns}
The celebrated theorems of Banach and Schauder, for instance, guarantee the existence of such fixed points under specific conditions, primarily concerning the operator's properties (e.g., contractivity, compactness) and the underlying functional space. \cite{Banach1922Operations, Schauder1930Fixpunkt}.

In the context of our investigation, the dual logarithmic derivative operator \(A\) exhibits specific periodic behaviour. We show that these period-2 orbits, equivalently the fixed points of \(A^{2}\), are precisely a two-parameter family of rational functions of the form \eqref{eq:all-period-2}, and the fixed point solutions, equivalently the period-1 orbits, are precisely in the form \eqref{eq:all-fixed-points}.
This simple yet exact result contrasts with the more complex behaviours often explored in general fixed-point theory, highlighting a distinct and precisely characterizable phenomenon.

\paragraph{Operator Dynamics.} 
Operator dynamics study the orbits produced by iterating a nonlinear operator on a function space.
\[
f, \quad A[f], \quad A^2[f], \quad A^3[f], \quad ...
\]
It serves as a powerful framework for studying the long-term behaviour of systems. 

In this spirit, the Dual Logarithmic Derivative operator induces a discrete dynamical system on rational functions.
Our identification of an exact period-2 cycle shows that its iteration does not lead to complex or chaotic behaviour, but instead collapses onto a simple and fully explicit periodic structure.
This makes the Dual Logarithmic Derivative operator a particularly transparent example within the broader landscape of operator-induced dynamics.

\paragraph{Our Contributions.}

\begin{itemize}
    \item We investigate the periodic behaviour of the dual logarithmic derivative operator $\mathcal{A}[f]\coloneqq \frac{\mathrm{d}\ln f}{\mathrm{d}\ln x}$ and show that it exhibits genuinely nondegenerate periodic orbits in the complex analytic setting.

    \item We identify a canonical explicit period-$2$ orbit of $\mathcal{A}$ and, motivated by this, obtain a complete classification of all nondegenerate period-$2$ orbits.  They are precisely the two-parameter rational pairs
    \[
        f_1(x)=\frac{c a x^{c}}{1-a x^{c}}, 
        \qquad
        f_2(x)=\frac{c}{1-a x^{c}},
    \]
    with $ac\neq 0$.

    \item We give a complete classification of all fixed points of $\mathcal{A}$: every such solution is of the form
    \[
        f(x)=\frac{1}{a-\ln x},
    \]
    on any domain where $a-\ln x\neq 0$.

    \item We show that logistic-type functions illustrate how non-periodic initial data may nevertheless flow into the classified period-$2$ family: after a logarithmic change of variables, they become pre-periodic and enter the period-$2$ orbit in a single iterate.
\end{itemize}

The structure of the paper follows the natural progression of the analysis.
Section~\ref{sec:definitions} introduces the operator $\mathcal{A}$ and the basic periodicity notions.
Section~\ref{sec:canonical-2cycle} presents a simple clean period-2 example, which motivates the classification developed in Section~\ref{sec:period2}.
There we show that every period-2 pair obeys a constant-difference constraint and derive the full family of solutions.
Section~\ref{sec:fixedpoints} provides the explicit form of all possible fixed point functions.
Section~\ref{sec:preperiodic} then illustrates pre-periodic behaviour through a logistic-type example.
Section~\ref{sec:futurework} outlines several directions for further study.

\section{Definitions}
\label{sec:definitions}

\subsection{The Operator $\mathcal{A}$}

\begin{definition}
\label{def:duallogderivativeoperator}
Let $\mathcal{A}$ denote the \emph{Dual Logarithmic Derivative operator} acting on complex-valued functions $f$ defined on a domain $U\subset \mathbb{C} \setminus \{0\}$ such that
\[
    f(x)\neq 0
    \quad\text{for all }x\in U,
\]
and such that $f$ is complex-differentiable on $U$.
For any fixed analytic branch of $\ln x$ and $\ln f(x)$ on $U$, We define
\[
    \mathcal{A}[f](x)
    \coloneqq \frac{\mathrm{d}(\ln f(x))}{\mathrm{d}(\ln x)}
    = \frac{x\, f'(x)}{f(x)} ,
\]
Equivalently,
\[
    \mathcal{A}[f]
    \coloneqq \frac{\mathrm{d}\ln f}{\mathrm{d}\ln x}
    = \frac{x f'}{f}.
\]
\end{definition}

\subsection{Period-2 Orbit}

\begin{definition}
\label{def:2cycle}
Let $\mathcal{A}$ be an operator acting on a space of holomorphic, nonvanishing functions on a domain $U\subset\mathbb{C} \setminus \{0\}$.
A pair $(f_1,f_2)$ is a \emph{period-$2$ orbit} of $\mathcal{A}$ if
\[
    \mathcal{A}[f_1] = f_2, 
    \qquad
    \mathcal{A}[f_2] = f_1 .
\]
The orbit is \emph{nondegenerate} if
\[
    f_1 \not\equiv f_2 .
\]
\end{definition}

\subsection{Fixed Point Function}

\begin{definition}
Let $\mathcal{A}$ be an operator acting on a function space.
A function $f$ is a \emph{fixed point}%
\footnote{Equivalently, a period-$1$ orbit of $\mathcal{A}$.}
of $\mathcal{A}$ if
\[
    \mathcal{A}[f]=f .
\]
\end{definition}

\section{A Canonical Period-2 Orbit}
\label{sec:canonical-2cycle}

Given the definitions, we present a simple \emph{nondegenerate period-$2$ orbit} of the dual logarithmic derivative operator:
\[
    \{f_1,f_2\}
    := 
    \left\{
        \frac{a x}{1-a x},
        \;\,
        \frac{1}{1-a x}
    \right\},
\]
and we do not distinguish between the labelings $(f_1,f_2)$ and $(f_2,f_1)$.

\begin{proposition}
\label{prop:canonical-2cycle}
Fix $a\in\mathbb{C}\setminus\{0\}$ and define
\[
    f_1(x)=\frac{a x}{1-a x},
    \qquad
    f_2(x)=\frac{1}{1-a x},
\]
for all $x$ in any domain $U\subset\mathbb{C}$ on which $1-a x\neq 0$.
Then $(f_1,f_2)$ forms a nondegenerate period-$2$ orbit of $\mathcal{A}$.
\end{proposition}

\begin{proof}
Differentiating gives
\[
    f_1'(x)=\frac{a}{(1-a x)^2},
    \qquad
    f_2'(x)=\frac{a}{(1-a x)^2}.
\]

Applying $\mathcal{A}$ to $f_1$,
\[
    \mathcal{A}[f_1]
    = \frac{x f_1'}{f_1}
    = \frac{x\cdot \tfrac{a}{(1-a x)^2}}
           {\tfrac{a x}{1-a x}}
    = \frac{1}{1-a x}
    = f_2 .
\]

Similarly,
\[
    \mathcal{A}[f_2]
    = \frac{x f_2'}{f_2}
    = \frac{x\cdot \tfrac{a}{(1-a x)^2}}
           {\tfrac{1}{1-a x}}
    = \frac{a x}{1-a x}
    = f_1 .
\]

Thus,
\[
    \mathcal{A}^2[f_1]
    =\mathcal{A}[\mathcal{A}[f_1]]
    =\mathcal{A}[f_2]
    =f_1,
\]
so $(f_1,f_2)$ is indeed a nondegenerate period-$2$ orbit of $\mathcal{A}$.
\end{proof}

\section{Classification of Periodic-2 Orbits}
\label{sec:period2}

In this section, we give a complete classification of nondegenerate period-2 orbits of the dual logarithmic derivative operator $\mathcal{A}$. In particular, we show that all such orbits lie in a subclass of rational functions in the form of \eqref{eq:all-period-2}.

\subsection{Period-2 Forces Constant Difference}

Our first step is to show that the two functions in any period-$2$ pair must differ by a constant.

\begin{lemma}
\label{lemma:constant-difference}
Let $(f_1,f_2)$ be a period-$2$ pair of $\mathcal{A}$ on a connected domain 
$U\subset\mathbb{C}\setminus\{0\}$, with $f_1$ and $f_2$ complex-differentiable and nowhere vanishing on $U$.
Then $f_2-f_1$ is constant on $U$.
\end{lemma}

\begin{proof}
From $\mathcal{A}[f_1]=f_2$ we have
\[
    f_1'(x)=\frac{f_1(x)\, f_2(x)}{x}.
\]
Similarly, $\mathcal{A}[f_2]=f_1$ yields
\[
    f_2'(x)=\frac{f_1(x)\, f_2(x)}{x}.
\]
Thus $f_1'(x)=f_2'(x)$ for all $x\in U$, and hence
\[
    (f_2 - f_1)'(x)=0.
\]
Since $U$ is connected, any complex-differentiable function with zero derivative must be constant on $U$.
Therefore, $f_2 - f_1 \equiv c$ for some $c\in\mathbb{C}$, and $c\neq 0$ because the orbit is nondegenerate.
\end{proof}

\subsection{Complete Classification of Periodic-2 Orbits}

\begin{theorem}
\label{thm:classification-period-2}

The set of all nondegenerate period-$2$ orbits of $\mathcal{A}$ is precisely the family
\begin{equation}
\keyeq\label{eq:all-period-2}
    f_1(x)=\frac{c\,a\, x^{c}}{1-a x^{c}},
    \qquad
    f_2(x)=\frac{c}{1-a x^{c}},
\end{equation}
for constants $a,c\in\mathbb{C}\setminus\{0\}$, defined on any connected domain
\[
    U \subset \mathbb{C}\setminus\{0\}
    \quad \text{on which} \quad
    1-a x^{c}\neq 0.
\]
so that both $f_1$ and $f_2$ are nowhere vanishing on $U$.
Moreover, the ordering of the pair is irrelevant: the unordered set
$\{f_1,f_2\}$ represents a single $2$-cycle under $\mathcal A$.

\end{theorem}

\begin{proof}
We first establish necessity and then sufficiency.

\medskip
\noindent\textit{Necessity.}

Suppose $(f_1,f_2)$ is a nondegenerate period-$2$ orbit.
Fix analytic branches of $\ln x$, $\ln f_1$, $\ln f_2$ on the connected domain $U\subset\mathbb{C}\setminus\{0\}$.
By Lemma~\ref{lemma:constant-difference}, there exists $c\in\mathbb{C}\setminus\{0\}$ such that
\begin{equation}\label{eq:f2=f1+c}
    f_2 = f_1 + c.
\end{equation}
Substituting \eqref{eq:f2=f1+c} into $\mathcal{A}[f_1]=f_2$ gives
\begin{equation}
    \mathcal{A}[f_1] = f_1 + c,
\end{equation}
that is,
\begin{equation}\label{eq:ode-f1-complex-raw}
    \frac{x f_1'}{f_1} = f_1 + c.
\end{equation}
Rewriting,
\begin{equation}\label{eq:ode-f1-complex}
    \frac{f_1'}{f_1(f_1+c)} = \frac{1}{x}.
\end{equation}

Since $x\neq 0$ and $f_1$, $f_1+c$ never vanish on $U$, the expression in \eqref{eq:ode-f1-complex} is well-defined, and we may separate variables:
\begin{equation}\label{eq:sep-variables}
    \int \frac{f_1'(x)}{f_1(x)\bigl(f_1(x)+c\bigr)}\,\mathrm{d}x
    \;=\;
    \int \frac{1}{x}\,\mathrm{d}x.
\end{equation}

Using
\[
    \frac{1}{f_1(f_1+c)}
    = \frac{1}{c}\!\left(\frac{1}{f_1} - \frac{1}{f_1+c}\right),
\]
the left-hand side of \eqref{eq:sep-variables} becomes
\begin{equation}
    \int \frac{f_1'}{f_1(f_1+c)}\,\mathrm{d}x
    = \frac{1}{c}\int\!\left(\frac{f_1'}{f_1} - \frac{f_1'}{f_1+c}\right)\mathrm{d}x.
\end{equation}
Thus \eqref{eq:sep-variables} yields
\begin{equation}
    \frac{1}{c}\int \frac{f_1'(x)}{f_1(x)}\,\mathrm{d}x
    \;-\;
    \frac{1}{c}\int \frac{f_1'(x)}{f_1(x)+c}\,\mathrm{d}x
    \;=\;
    \int \frac{1}{x}\,\mathrm{d}x.
\end{equation}

Integrating and absorbing constants into a single $K\in\mathbb{C}$, we obtain
\begin{equation}
    \frac{1}{c}\bigl(\ln f_1 - \ln(f_1+c)\bigr)
    = \ln x + K.
\end{equation}
Equivalently,
\begin{equation}
    \ln\!\left(\frac{f_1}{f_1+c}\right)
    = c \ln x + K,
\end{equation}
and hence
\begin{equation}
\label{eq:f1-over-f1+c}
    \frac{f_1}{f_1+c}
    = a x^{c},
    \qquad a = e^{K}\in\mathbb{C}\setminus\{0\}.
\end{equation}

Solving \eqref{eq:f1-over-f1+c} for $f_1$,
\begin{equation}
\label{eq:period2-f1-complex}
    f_1(x)
    = \frac{c a x^{c}}{1 - a x^{c}}.
\end{equation}
Then from $f_2=f_1+c$ we obtain
\begin{equation}
\label{eq:period2-f2-complex}
    f_2(x)
    = \frac{c}{1 - a x^{c}}.
\end{equation}

Thus, any nondegenerate period-$2$ orbit must be of the form
\eqref{eq:period2-f1-complex}–\eqref{eq:period2-f2-complex}
for some $a,c\in\mathbb{C}\setminus\{0\}$ on a domain where $1-a x^{c}\neq 0$.

\medskip
\noindent\textit{Sufficiency.}

Conversely, let $a,c\in\mathbb{C}\setminus\{0\}$ and define
\begin{equation}
    f_1(x)=\frac{c a x^{c}}{1-a x^{c}},
    \qquad
    f_2(x)=\frac{c}{1-a x^{c}},
\end{equation}
on a domain $U$ where $1-a x^{c}\neq 0$.  
Then $f_2 = f_1 + c$ holds identically.

We have
\begin{equation}
    \ln f_1(x)
    = \ln c + \ln a + c\ln x - \ln(1-a x^{c}),
\end{equation}
and hence
\begin{equation}
    \mathcal{A}[f_1](x)
    = \frac{\mathrm{d}(\ln f_1)}{\mathrm{d}(\ln x)}
    = c + \frac{x\,( a c x^{c-1})}{1-a x^{c}}
    = c + \frac{c a x^{c}}{1-a x^{c}}
    = \frac{c}{1-a x^{c}}
    = f_2(x).
\end{equation}

Similarly,
\begin{equation}
    \ln f_2(x)
    = \ln c - \ln(1-a x^{c}),
\end{equation}
so
\begin{equation}
    \mathcal{A}[f_2](x)
    = \frac{\mathrm{d}(\ln f_2)}{\mathrm{d}(\ln x)}
    = \frac{x\,( a c x^{c-1})}{1-a x^{c}}
    = \frac{c a x^{c}}{1-a x^{c}}
    = f_1(x).
\end{equation}

Therefore, $(f_1,f_2)$ is indeed a period-$2$ orbit of $\mathcal{A}$.

\medskip
This completes the proof.
\end{proof}

\section{Classification of Fixed Points}
\label{sec:fixedpoints}

In this section, we classify all fixed points of the dual logarithmic derivative operator $\mathcal{A}$.  
We will show that every fixed point must belong to a specific family of rational functions described in \eqref{eq:all-fixed-points}.

\subsection{Constant Functions Are Not Fixed Points}

\begin{lemma}\label{lemma:no-constant-fixed}
No non-zero constant function can satisfy $\mathcal{A}[f]=f$ on any domain 
$U\subset \mathbb{C}\setminus\{0\}$.
\end{lemma}

\begin{proof}
Let $f(x)\equiv c$ with $c\in\mathbb{C}\setminus\{0\}$.  
Then $f'(x)\equiv 0$, so
\[
    \mathcal{A}[f](x)
    = \frac{x f'(x)}{f(x)}
    \equiv 0.
\]
Since $c\neq 0$, we have $\mathcal{A}[f]\not\equiv f$.  
Thus no non-zero constant function is a fixed point of $\mathcal{A}$.
\end{proof}

\subsection{Smoothness of Fixed Points}

The fixed point equation forces strong regularity.

\begin{lemma}\label{lemma:Cinfty-fixed}
Let $U \subset \mathbb{C}\setminus\{0\}$ be a connected domain.
If $f\colon U\to\mathbb{C}\setminus\{0\}$ is complex-differentiable on $U$, and satisfies 
$\mathcal{A}[f]=f$ on $U$, then $f \in C^\infty(U)$.
\end{lemma}

\begin{proof}
Since $\mathcal{A}[f](x)=\frac{x f'(x)}{f(x)}$ by definition, the equation $\mathcal{A}[f]=f$ implies that $f$ is differentiable on $U$ and satisfies
\begin{equation}\label{eq:fixedpoint-ODE}
    f'(x)=\frac{f(x)^2}{x},
    \qquad x\in U.
\end{equation}

We show by induction that $f\in C^{k}(U)$ for every $k\ge 0$.

\medskip
\noindent\textit{Base step (\(k=0\))}:

Since $f$ is differentiable, it is continuous; hence $f\in C^{0}(U)$.

\medskip
\noindent\textit{Inductive step}:

Assume $f\in C^{k}(U)$.  

Let $g(x)=\frac{1}{x}$; then $g\in C^\infty(U)$, and the right-hand side of \eqref{eq:fixedpoint-ODE},
\[
    f(x)^2\, g(x),
\]
is a product of functions in $C^k(U)$, hence also lies in $C^k(U)$.
Thus,
\begin{equation}
    f'(x)=f(x)^2 g(x) \in C^k(U).
\end{equation}
This implies
\begin{equation}
    f \in C^{k+1}(U).
\end{equation}

By induction, $f\in C^{k}(U)$ for all $k\ge 0$, and hence $f\in C^\infty(U)$.
\end{proof}

\subsection{Complete Classification of Fixed Points}

We now present the complete classification of all fixed points of $\mathcal{A}$.

\begin{theorem}
\label{thm:fixed-points}
A nowhere-vanishing complex-differentiable function $f$ on a domain 
$U\subset \mathbb{C}\setminus\{0\}$ is a fixed point of $\mathcal{A}$,
\[
    \mathcal{A}[f]=f,
\]
if and only if it is of the form
\begin{equation}
\keyeq\label{eq:all-fixed-points}
    f(x)=\frac{1}{\,a-\ln x\,},
\end{equation}
for some constant $a\in\mathbb{C}$, on any domain $U$ on which $a-\ln x\neq 0$.
\end{theorem}

\begin{proof}
We prove necessity and sufficiency.

\medskip
\noindent\textit{Necessity.}

Assume $\mathcal{A}[f]=f$.
Fix analytic branches of $\ln x$, $\ln f$ on a connected domain $U\subset\mathbb{C}\setminus\{0\}$.
Then
\[
    \frac{x f'}{f}=f
    \qquad\Longleftrightarrow\qquad
    \frac{f'}{f^{2}}=\frac{1}{x}.
\]

Since $f$ is differentiable and nowhere zero, the change of variables
$df=f'(x)\,dx$ is valid, and we may separate variables:
\begin{equation}\label{eq:fixed-sep}
    \int \frac{1}{f(x)^{2}}\, df
    \;=\;
    \int \frac{1}{x}\, dx.
\end{equation}

Because
\[
    \int f^{-2}\,df=-\frac{1}{f},
\]
integration of \eqref{eq:fixed-sep} yields
\begin{equation}
    -\frac{1}{f(x)}=\ln x + C,
    \qquad C\in\mathbb{C}.
\end{equation}
Setting $a=-C$ gives the general form
\begin{equation}\label{eq:fixedpoint-form}
    f(x)=\frac{1}{\,a-\ln x\,}.
\end{equation}
The only domain restriction is that $a-\ln x\neq 0$, so that $f$ is defined and
nowhere vanishing on $U$.

\medskip
\noindent\textit{Sufficiency.}

Conversely, suppose $a\in\mathbb{C}$ and define
\[
    f(x)=\frac{1}{a-\ln x}
\]
on a domain $U$ with $a-\ln x\neq 0$.  
Then
\[
    f'(x)
    = \frac{1}{(a-\ln x)^{2}} \cdot \frac{1}{x}
    = \frac{f(x)^{2}}{x}.
\]
Hence
\[
    \mathcal{A}[f](x)
    = \frac{x f'(x)}{f(x)}
    = f(x),
\]
so $f$ is indeed a fixed point.

\medskip
This completes the classification.
\end{proof}

\section{Pre-Periodic Behaviour Under the Operator $\mathcal{A}$}
\label{sec:preperiodic}

Not all initial functions lie on a genuine periodic orbit of $\mathcal{A}$.
Nevertheless, many functions exhibit a pre-periodic pattern: after a finite number of iterates, they enter the classified period-$1$ or period-$2$ families.
This phenomenon shows that the operator’s dynamics extend beyond its exact periodic solutions, and that the known orbits attract a broader class of inputs in finitely many steps.

A representative example is provided by logistic-type functions, which are not periodic under $\mathcal{A}$ but are empirically observed to reach a $2$-cycle after a few iterations.
We include this case to illustrate the existence of pre-periodic behaviour, without attempting a full characterization.

\paragraph{Example: Logistic-Type Functions}

Sigmoidal functions, and the logistic curve in particular, arise throughout statistics, biology, and machine learning.  When expressed in logarithmic coordinates, the logistic function becomes a simple rational expression, placing it within the broader class of functions on which the operator $\mathcal{A}$ acts in a particularly transparent manner.  Although the logistic function is not itself a member of the period-$2$ family classified in Section~\ref{sec:period2}, its orbit under~$\mathcal{A}$ displays a  \emph{one-step pre-periodic} behaviour: its orbit enters the period-$2$ family after a single iteration and then follows the genuine $2$-cycle exactly.  This makes the logistic curve a natural illustrative example for understanding how standard sigmoidal functions behave within the log--log differential framework.

\medskip
\noindent\emph{Log-Domain Representation}
\medskip

Let $\sigma$ denote the logistic function.
\begin{equation}
    \sigma(t)=1/(1+e^{-t})
\end{equation}

Under the change of variables $t=\ln x$ (so $x=e^t$), we obtain the
rational form
\begin{equation}
    \sigma(\ln x)=\frac{x}{1+x}, \qquad x>0.
\end{equation}

This representation enables a direct analysis of its orbit under~$\mathcal{A}$.

\medskip
\noindent\emph{Period-2 Orbit and One-Step Pre-Periodic Behaviour}
\medskip

Let $f(x) = \sigma(\ln x)$.
A straightforward computation gives the first iterate:
\[
    g_1(x) := \mathcal{A}[f](x)
            = \frac{x f'(x)}{f(x)}
            = 1 - f(x)
            = \frac{1}{1+x}.
\]

Applying $\mathcal{A}$ once more yields
\[
    g_2(x) := \mathcal{A}[g_1](x)
            = \frac{x g_1'(x)}{g_1(x)}
            = -\frac{x}{1+x}.
\]

Since $g_2$ is nowhere zero on the chosen domain $U$, the operator $\mathcal A$ remains well-defined (after fixing branches of $\ln x$ and $\ln g_2(x)$).
\[
    \mathcal{A}[g_2](x)=\frac{x g_2'(x)}{g_2(x)} = \frac{1}{1+x} = g_1(x),
\]

So the pair $(g_2,g_1)$ forms a genuine period-$2$ orbit.

In particular, the logistic-induced function $f(x)=x/(1+x)$ is
\emph{pre-periodic} under $\mathcal{A}$: it is not itself in the period-$2$
family, but its first iterate $g_1$ already lies in that family, and the
subsequent trajectory follows the exact $2$-cycle between $g_1$ and $g_2$.


\section{Future Work}
\label{sec:futurework}

Several natural directions remain open.

\paragraph{Higher-period behaviour.}
While period-$1$ and period-$2$ orbits are completely classified, it remains open whether genuine period-$n$ solutions exist for $n\ge 3$, and if so whether they admit tractable closed forms. Resolving this would sharpen the global picture of the operator’s dynamics.

\paragraph{Pre-periodic phenomena.}
Many functions appear to become periodic only after several iterations. Characterizing which initial functions eventually fall into the known period-$1$ or period-$2$ families, and understanding the structure of their pre-periods, is a natural next step.

\paragraph{Broader operator families.}
Investigating operators related to $\mathcal{A}$ that exhibit similar periodic or pre-periodic patterns may reveal connections with broader themes in functional equations and operator-induced dynamics.

\paragraph{Potential applications.}
Although $\mathcal{A}$ has a natural log--log differential origin, its applications remain largely open. Identifying settings—such as growth laws or scaling analyses—where the operator captures useful structure would extend its relevance.

\bibliographystyle{acm}
\bibliography{ref}

\end{document}